\documentclass[11pt]{article}
\usepackage{amsmath}
\usepackage{amssymb}
\usepackage{amsthm}
\usepackage{bookmark}
\usepackage{geometry}
\usepackage{hyperref}
\newtheorem{theorem}{Theorem}
\newtheorem{lemma}[theorem]{Lemma}
\newtheorem{proposition}[theorem]{Proposition}

\DeclareMathOperator{\N}{N}
\DeclareMathOperator{\Tr}{Tr}

\title{Trinomial Planar Functions on Cubic and Quartic Extensions of Finite Fields}
\author{Ruikai Chen\textsuperscript{1,2}\and Sihem Mesnager\textsuperscript{1,2,3}}
\date{\small\textsuperscript{1}Department of Mathematics, University of Paris VIII, F-93526 Saint-Denis\\\textsuperscript{2}Laboratory Analysis, Geometry and Applications, LAGA, University Sorbonne Paris Nord, CNRS, UMR 7539, F-93430, Villetaneuse, France\\\textsuperscript{3}Telecom Paris, Polytechnic institute of Paris, 91120 Palaiseau, France\\Emails: \href{mailto:chen.rk@outlook.com}{chen.rk@outlook.com}\quad\href{mailto:smesnager@univ-paris8.fr}{smesnager@univ-paris8.fr}}

\begin{document}

\maketitle

\begin{abstract}
Planar functions, introduced by Dembowski and Ostrom, are functions from a finite field to itself that give rise to finite projective planes. They exist, however, only for finite fields of odd characteristics. They have attracted much attention in the last decade thanks to their interest in theory and those deep and various applications in many fields. This paper focuses on planar trinomials over cubic and quartic extensions of finite fields. Our achievements are obtained using connections with quadratic forms and classical algebraic tools over finite fields. Furthermore, given the generality of our approach, the methodology presented could be employed to drive more planar functions on some finite extension fields.
\end{abstract}


\section{Introduction}

Consider an extension $\mathbb F_{q^n}/\mathbb F_q$ of finite fields of odd characteristic. The function induced by a polynomial $f$ over $\mathbb F_{q^n}$ is called bent, if for every $b\in\mathbb F_{q^n}$,
\[\left|\sum_{t\in\mathbb F_{q^n}}\psi(f(t)-bt)\right|^2=q^n,\]
where $\psi$ is the canonical additive character of $\mathbb F_{q^n}$. In addition, it is called a planar function if $f(x+c)-f(x)$ is a permutation on $\mathbb F_{q^n}$ for every $c\in\mathbb F_{q^n}^*$. As a well-known result, $f(x)$ is planar if and only if $cf(x)$ is bent for every $c\in\mathbb F_{q^n}^*$. Constructions of planar functions can be found in \cite{coulter2008commutative,coulter1997planar, ding2006family,zha2009perfect}.

Historically, planar functions with odd characteristics have been introduced by Dembowski and Ostrom (\cite{Dembowski-Ostrom68}), aiming to design finite projective planes with odd characteristics. These functions are also connected with important mathematical objects, such as permutation polynomials over finite fields. The notion of planar functions coincides with the one of perfect nonlinear (PN) functions from $\mathbb F_q$ to itself (see, e.g., \cite[Chapter 6, pp. 296]{Carlet-2021}, \cite{Pott-2016} and \cite{Zieve}). Planar functions have deep applications in different areas of mathematics for communication (coding theory, cryptography, and related topics such as combinatorics).

In particular, it has been shown that planar functions are helpful in DES-like cryptosystems (\cite{Nyberg-Knudsen93}). In addition, it has been demonstrated in several articles (see, e.g., \cite{Carlet-Ding-Yuan-2005,Ding-et-al-2009,Yuan-et-al-2006}) that planar functions allow constructing error-correcting codes, which are then employed to design secret sharing schemes. Planar functions are also applied to the constructions of authentication codes (\cite{Ding-Niederreiter}), constant composition codes (\cite{Ding-et-al-2005}) and signal sets (\cite{Ding-et-al-2007}). Besides, planar functions induce many combinatorial objects such as skew Hadamard difference sets, Paley type partial difference sets (\cite{Weng}), relative difference sets (\cite{Ganley-1975}) and symplectic spreads (\cite{Abdukhalikov}). Also, it has been shown in two main papers \cite{Zhou-2013} and \cite{Schmidt-Zhou-2014} that a natural analogue notion of planar functions over fields of characteristic $2$ could be introduced and studied, which is sometimes called pseudo-planar or modified planar (see \cite[Chapter 6, pp. 296]{Carlet-2021}). Such planar functions have similar properties concerning relative difference sets and finite geometries as their counterparts in odd characteristics. Finally, by employing various methods, much attention has been addressed to the problem of the existence and non-existence of planar mappings. For example, the classification of planar monomials over fields of prime square order has been obtained by Coulter (\cite{coulter2006}). Very recent results on this framework are given by Bergman et al. (\cite{Bergman-2022}) devoted to classifying planar monomials over fields of order a prime cubed, and by Beierle and Felke toward the classification of planar monomials over finite fields of small order (\cite{Beierle-Felke-2022,Beierle-Felke-2022-bis}).

Note that most of known planar functions on $\mathbb F_{q^n}$ are of the form
\[f(x)=\sum_{0\le j\le i<n}a_{ij}x^{q^i+q^j},\]
also known as Dembowski-Ostrom polynomials. Some work has been done on this class of polynomials. For example, those planar binomial functions on $\mathbb F_{q^3}$ are classified in \cite{kyuregyan2012some}. Also, the sufficient and necessary conditions for the polynomial $x^{q^2+1}+x^{q+1}+ux^2$ ($u\in\mathbb F_{q^3}$) to be planar are given in \cite{kyureghyan2012planarity} and \cite{coulter2013conjecture}, as well as the polynomial $x^{q^2+1}+ax^{q+1}+bx^2$ considered in \cite{Bartoli-Bonini-2022}, however, with coefficients restricted in $\mathbb F_q$.

In this paper, we will extend these results. In Section \ref{cubic}, we provide characterizations of two classes of planar trinomials on $\mathbb F_{q^3}$. The polynomials from \cite{kyureghyan2012planarity}, \cite{coulter2013conjecture} and \cite{Bartoli-Bonini-2022} are special cases of those in Theorem \ref{maintheo2}. In Section \ref{quartic}, we construct planar trinomials on $\mathbb F_{q^4}$. The method used there can be applied to more planar functions on extensions of finite fields.

\section{An Overview of the Approach}

Consider the polynomial ring $\mathbb F_{q^n}[x]$ with an indeterminate $x$. When discussing functions on $\mathbb F_{q^n}$, we identify them with polynomials in $\mathbb F_{q^n}[x]$ modulo $x^{q^n}-x$. We focus on planar functions on $\mathbb F_{q^n}$ are of the form
\[f(x)=\sum_{0\le j\le i<n}a_{ij}x^{q^i+q^j}.\]
In particular, the monomial $x^{q^k+1}$ is planar if $\frac n{\gcd(k,n)}$ is odd. For the general case, we can investigate $cf(x)$ for every $c\in\mathbb F_{q^n}^*$. In fact, $\Tr(cf(x))$ is a quadratic form on $\mathbb F_{q^n}$ as a vector space over $\mathbb F_q$, where $\Tr$ denotes the trace function from $\mathbb F_{q^n}$ to $\mathbb F_q$; moreover,
\[\Tr(cf(x))=\Tr\left(\sum_{0\le j\le i<n}(ca_{ij})^{q^{n-j}}x^{q^{i-j}+1}\right).\]
It has been proved (e.g., in \cite[Proposition 2]{helleseth2006monomial}) that $cf(x)$ is a bent function on $\mathbb F_{q^n}$ if and only if the corresponding quadratic form is nondegenerate, if and only if
\[\left|\sum_{t\in\mathbb F_{q^n}}\psi(cf(t))\right|^2=q^n.\]
To determine whether $f$ is a planar function, we calculate the determinant of the quadratic form $\Tr(cf(x))$ for every $c\in\mathbb F_{q^n}^*$. Specifically, it can be characterized given the coefficients as follows.

\begin{lemma}\label{nondegenerate}
Let $f(x)=ax^{q+1}+bx^2$ for some $a,b\in\mathbb F_{q^n}^*$. Then
\[\left|\sum_{t\in\mathbb F_{q^n}}\psi(f(t))\right|^2=q^n\]
if and only if
\begin{itemize}
\item $\Tr(a)^2-4\N(b)\ne0$ in the case $n=2$, and
\item $4\N(b)+\N(a)-\Tr\big(a^2b^{q^2}\big)\ne0$ in the case $n=3$,
\end{itemize}
where $\Tr$ and $\N$ are the trace and the norm function from $\mathbb F_{q^n}$ to $\mathbb F_q$ respectively.
\end{lemma}
\begin{proof}
This follows from \cite[Proposition 2]{helleseth2006monomial} or \cite[Section 3]{chen2023evaluation}, where the corresponding matrix is
\[\begin{pmatrix}2b&a^q+a\\a^q+a&2b^q\end{pmatrix}\]
in the case $n=2$, and
\[\begin{pmatrix}2b&a&a^{q^2}\\a&2b^q&a^q\\a^{q^2}&a^q&2b^{q^2}\end{pmatrix}\]
in the case $n=3$.
\end{proof}

There is an equivalence relation preserving the properties of planar functions. Two Dembowski-Ostrom polynomials $f_0$ and $f_1$ over $\mathbb F_{q^n}$ are called equivalent if there exist linearized permutation polynomial $L_0$ and $L_1$ over $\mathbb F_{q^n}$ such that $f_0(x)=L_0(f_1(L_1(x)))$. Note that all those planar functions on $\mathbb F_{q^3}$ are equivalent to either $x^2$ or $x^{q+1}$ according to \cite[Lemma 2.1]{coulter2013conjecture}.

The following result will be useful for the remaining part.

\begin{lemma}[{\cite[Theorem 5.3]{moisio2008kloosterman}}]\label{tr_n}
For any $r\in\mathbb F_q$ and $s\in\mathbb F_q^*$, there exists an element $\alpha\in\mathbb F_{q^3}^*$ with $\Tr(\alpha)=r$ and $\N(\alpha)=s$. In other words, given $a\in\mathbb F_{q^3}^*$, the function $\Tr(ax^{q-1})$ takes all values in $\mathbb F_q$.
\end{lemma}

\section{Planar Trinomials on Cubic Extensions}\label{cubic}

Let $\Tr$ and $\N$ denote the trace and the norm function from $\mathbb F_{q^3}$ to $\mathbb F_q$ respectively. We characterize two classes of trinomial planar functions in what follows.

\begin{proposition}\label{AB}
For $A,B\in\mathbb F_{q^3}^*$ and $r\in\mathbb F_q$, the function $\Tr(Ax^{q-1}+Bx^{1-q})+r$ has no zero in $\mathbb F_{q^3}^*$ if and only if $r=\frac{\N(A)+\N(B)}{AB}\ne0$.
\end{proposition}
\begin{proof}
The equation $\Tr(Ax^{q-1}+Bx^{1-q})+r=0$ is equivalent to $Ax^{q-1}+B^{q^2}x^{q^2-1}+u=0$ and $\Tr(u)=r$ for some $u\in\mathbb F_{q^3}$. Here $Ax^{q-1}+B^{q^2}x^{q^2-1}+u$ has a zero in $\mathbb F_{q^3}^*$ if and only if $B^{q^2}x^{q^2}+Ax^q+ux$ does not permute $\mathbb F_{q^3}$; that is,
\[\N(u)+\N(A)+\N(B)-\Tr\big(ABu^{q^2}\big)=\begin{vmatrix}u&A&B^{q^2}\\B&u^q&A^q\\A^{q^2}&B^q&u^{q^2}\end{vmatrix}=0.\]
Therefore, the function $\Tr(Ax^{q-1}+Bx^{1-q})+r$ has a zero in $\mathbb F_{q^3}^*$ if and only if
\begin{equation}\label{tr_u}\Tr(u)=r\quad\text{and}\quad\N(u)+\N(A)+\N(B)-\Tr\big(ABu^{q^2}\big)=0\end{equation}
for some $u\in\mathbb F_{q^3}$.

Suppose $AB\in\mathbb F_q$. If $r=\frac{\N(A)+\N(B)}{AB}\ne0$, then $\Tr(u)=r$ implies that $u\ne0$ and
\[\begin{split}&\mathrel{\phantom{=}}\N(u)+\N(A)+\N(B)-\Tr\big(ABu^{q^2}\big)\\&=\N(u)+\N(A)+\N(B)-ABr\\&=\N(u)\ne0.\end{split}\]
If $r=\frac{\N(A)+\N(B)}{AB}=0$, then $u=0$ implies \eqref{tr_u}. If $r\ne\frac{\N(A)+\N(B)}{AB}$, then by Lemma \ref{tr_n}, there exists $u\in\mathbb F_{q^3}$ with $\Tr(u)=r$ and
\[\N(u)=-\N(A)-\N(B)+ABr=-\N(A)-\N(B)+\Tr\big(ABu^{q^2}\big),\]
which means $\eqref{tr_u}$ holds.

Suppose $AB\notin\mathbb F_q$ and let $\beta_0,\beta_1,\beta_2$ be the dual basis of $1,AB,(AB)^2$ in $\mathbb F_{q^3}/\mathbb F_q$. For arbitrary $r\in\mathbb F_q$, if $u^{q^2}=r\beta_0+r_1\beta_1+r_2\beta_2$ for some $r_1,r_2\in\mathbb F_q$, then $\Tr(u)=r$ and
\[\N(u)+\N(A)+\N(B)-\Tr\big(ABu^{q^2}\big)=\N(r\beta_0+r_1\beta_1+r_2\beta_2)+\N(A)+\N(B)-r_1.\]
It remains to show that the right side is zero for some $r_1,r_2\in\mathbb F_q$. Let $y$ be an element in some extension of $\mathbb F_q(x)$ with
\begin{equation}\label{1}\N^\prime(r\beta_0+x\beta_1+y\beta_2)+\N(A)+\N(B)-x=0,\end{equation}
where
\[\begin{split}&\mathrel{\phantom{=}}\N^\prime(r\beta_0+x\beta_1+y\beta_2)\\&=(r\beta_0+x\beta_1+y\beta_2)(r\beta_0^q+x\beta_1^q+y\beta_2^q)\big(r\beta_0^{q^2}+x\beta_1^{q^2}+y\beta_2^{q^2}\big).\end{split}\]
Assume $y\in\overline{\mathbb F_q}(x)$ and let $v_\infty$ be the valuation at the infinite place of $\overline{\mathbb F_q}(x)$. If $v_\infty(y)>-1$, then $v_\infty(\N^\prime(r\beta_0+x\beta_1+y\beta_2))=-3$. If $v_\infty(y)<-1$, then $v_\infty(\N^\prime(r\beta_0+y\beta_1+x\beta_2))=3v_\infty(y)$. Both cases contradict the fact $v_\infty(x-\N(A)-\N(B))=-1$. Thus we have $v_\infty(y)=-1$. Note that $y$ is integral over $\overline{\mathbb F_q}[x]$, which is an integrally closed domain. Then $y=\lambda x+\mu$ for some $\lambda,\mu\in\overline{\mathbb F_q}$, and
\[\begin{split}&\mathrel{\phantom{=}}\N^\prime(r\beta_0+x\beta_1+y\beta_2)\\&=((\beta_1+\lambda\beta_2)x+r\beta_0+\mu\beta_2)((\beta_1^q+\lambda\beta_2^q)x+r\beta_0^q+\mu\beta_2^q)\big(\big(\beta_1^{q^2}+\lambda\beta_2^{q^2}\big)x+r\beta_0^{q^2}+\mu\beta_2^{q^2}\big).\end{split}\]
Since $v_\infty(\N^\prime(r\beta_0+x\beta_1+y\beta_2))=-1$, without loss of generality we have
\[\beta_1+\lambda\beta_2=\beta_1^q+\lambda\beta_2^q=0.\]
Then $\lambda\in\mathbb F_q$, but $\beta_1$ and $\beta_2$ are linearly independent over $\mathbb F_q$, a contradiction. This shows that $\mathbb F_q(x,y)/\mathbb F_q(x)$ is an extension of function fields of degree $3$ with constant field $\mathbb F_q$.

Consider the corresponding plane projective curve defined by \eqref{1}, which is absolutely irreducible, for a cubic polynomial reducible over a field must have a root in it (see also \cite[Corollary 3.6.8]{stichtenoth2009}). As a consequence of \cite{aubry1996weil}, the number $N$ of rational points over $\mathbb F_q$ of the curve satisfies
\[|N-q-1|\le(3-1)(3-2)\sqrt q=2\sqrt q,\]
which means
\[N\ge q+1-2\sqrt q>0,\]
as $q\ge3$. Clearly, there is no point at infinity on this curve, so there exists at least one element in $\mathbb F_{q^3}$ with desired properties.
\end{proof}

\begin{theorem}\label{maintheo2}
For $a,b\in\mathbb F_{q^3}^*$, the function $x^{q^2+1}+ax^{q+1}+bx^2$ is planar on $\mathbb F_{q^3}$ if and only if $a=b^{q+1}$ with $\N(b)\ne1$, or $\N(a)-2ab^{q^2}+1=0$ with $\N(a)^2\ne1$.
\end{theorem}
\begin{proof}
For $c\in\mathbb F_{q^3}^*$, note that
\[\Tr\big(c\big(x^{q^2+1}+ax^{q+1}+bx^2\big)\big)=\Tr((c^q+ac)x^{q+1}+bcx^2),\]
and
\[\begin{split}&\mathrel{\phantom{=}}4\N(bc)+\N(c^q+ac)-\Tr\big((c^q+ac)^2(bc)^{q^2}\big)\\&=4\N(bc)+\N(c)+\N(ac)+\Tr\big(ac^{q^2+2}\big)+\Tr(a^{q+1}c^{q+2})\\&\mathrel{\phantom{=}}-\Tr\big(b^{q^2}c^{q^2+2q}+2ab^{q^2}c^{q^2+q+1}+a^2b^{q^2}c^{q^2+2}\big)\\&=\N(c)\big(4\N(b)+1+\N(a)-2\Tr\big(ab^{q^2}\big)\big)\\&\mathrel{\phantom{=}}+\N(c)\Tr\big(\big(a^{q^2+q}-b^{q^2}\big)c^{q-1}+\big(a-a^2b^{q^2}\big)c^{1-q}\big).\end{split}\]
By Lemma \ref{nondegenerate}, the function is planar if and only if
\[\Tr(Ac^{q-1}+Bc^{1-q})+4\N(b)+1+\N(a)-2\Tr\big(ab^{q^2}\big)\ne0\]
for all $c\in\mathbb F_{q^3}^*$, where $A=a^{q^2+q}-b^{q^2}$ and $B=a-a^2b^{q^2}$. This happens only if $AB\ne0$, for otherwise either $AB=0\ne A+B$, or $A=B=0\ne4\N(b)+1+\N(a)-2\Tr\big(ab^{q^2}\big)$. The former is excluded in view of Lemma \ref{tr_n}, and the latter implies $4\N(b)+1+\N(a)-2\Tr\big(ab^{q^2}\big)=0$ by a simple investigation. Then that is equivalent to
\begin{equation}\label{AB}AB\ne0\quad\text{and}\quad\frac{\N(A)+\N(B)}{AB}=4\N(b)+1+\N(a)-2\Tr\big(ab^{q^2}\big)\end{equation}
by Proposition \ref{AB}.

We claim that if $\delta=ab^{q^2}\in\mathbb F_q$, then
\[\begin{split}&\mathrel{\phantom{=}}\N(A)+\N(B)-AB\big(4\N(b)+1+\N(a)-2\Tr\big(ab^{q^2}\big)\big)\\&=\N(a)^{-1}\delta(\delta^2-\N(a))(\N(a)-2\delta+1)^2.\end{split}\]
In this case $A=a^{-1}(\N(a)-\delta)$ and $B=a(1-\delta)$. Then
\[\N(A)+\N(B)=\N(a)^{-1}(\N(a)-\delta)^3+\N(a)(1-\delta)^3,\]
and
\[\begin{split}&\mathrel{\phantom{=}}AB\big(4\N(b)+1+\N(a)-2\Tr\big(ab^{q^2}\big)\big)\\&=(\N(a)-\delta)(1-\delta)(4\N(a)^{-1}\delta^3+1+\N(a)-6\delta)\\&=\N(a)^{-1}(\N(a)-\delta)(1-\delta)(4\delta^3+\N(a)+\N(a)^2-6\N(a)\delta).\end{split}\]
It is routine to expand
\[(\N(a)-\delta)^3+\N(a)^2(1-\delta)^3-(\N(a)-\delta)(1-\delta)(4\delta^3+\N(a)+\N(a)^2-6\N(a)\delta)\]
and
\[\delta(\delta^2-\N(a))(\N(a)-2\delta+1)^2,\]
which are actually equal.

Assume that \eqref{AB} holds. Then $(\N(a)-\delta)(1-\delta)=AB\in\mathbb F_q$, which means $\delta\in\mathbb F_{q^2}\cap\mathbb F_{q^3}=\mathbb F_q$. By the above claim we have $\delta^2=\N(a)$ or $\N(a)-2\delta+1=0$. If $\delta^2=\N(a)$, then
\[a=\frac{\N(a)}{a^{q^2+q}}=\frac{\big(ab^{q^2}\big)^{q^2+q}}{a^{q^2+q}}=b^{q+1},\]
and $\N(b)=ab^{q^2}\ne1$ as $AB\ne0$. If $\N(a)-2\delta+1=0$, then $\N(a)-1=2(\delta-1)\ne0$ and $\N(a)+1=2\delta\ne0$. For the converse, if $a=b^{q+1}$ and $\N(b)\ne1$, then $\delta=\N(b)\in\mathbb F_q^*$, and $\delta^2=\N(b)^2=\N(a)$ with $AB=(\N(b)^2-\N(b))(1-\N(b))\ne0$. If $\N(a)-2ab^{q^2}+1=0$ and $\N(a)^2\ne1$, then $2\delta=\N(a)+1\in\mathbb F_q^*$ and $4AB=(\N(a)-1)(1-\N(a))\ne0$. Then the proof is complete.
\end{proof}

Using the same argument, we characterize another class of trinomials.

\begin{theorem}
For $a,b\in\mathbb F_{q^3}^*$, the function $x^{q+1}+ax^{2q}+bx^2$ is planar on $\mathbb F_{q^3}$ if and only if $4ab=1$ and $\N(2a)\ne-1$, or $4ab\ne1$ and $\N(a)+\N(b)+ab=0$.
\end{theorem}
\begin{proof}
For $c\in\mathbb F_{q^3}^*$, note that
\[\Tr(c(x^{q+1}+ax^{2q}+bx^2))=\Tr\big(cx^{q+1}+\big(a^{q^2}c^{q^2}+bc\big)x^2\big),\]
and
\[\begin{split}&\mathrel{\phantom{=}}4\N\big(a^{q^2}c^{q^2}+bc\big)+\N(c)-\Tr\big(c^2\big(a^{q^2}c^{q^2}+bc\big)^{q^2}\big)\\&=4\big(\N(ac)+\N(bc)+\Tr\big((ac)^{q^2+q}(bc)^q\big)+\Tr\big(ac(bc)^{q^2+1}\big)\big)\\&\mathrel{\phantom{=}}+\N(c)-\Tr\big(a^{q^2}c^{q^2+2q}+b^{q^2}c^{q^2+2}\big)\\&=4\N(c)\big(\N(a)+\N(b)+\Tr\big(a^{q^2+q}b^qc^{q-1}\big)+\Tr\big(ab^{q^2+1}c^{1-q}\big)\big)\\&\mathrel{\phantom{=}}+\N(c)\big(1-\Tr\big(a^{q^2}c^{q-1}+b^{q^2}c^{1-q}\big)\big).\end{split}\]
Clearly the function is planar if and only if
\[\Tr(Ac^{q-1}+Bc^{1-q})+4\N(a)+4\N(b)+1\ne0\]
for all $c\in\mathbb F_{q^3}^*$, where
\[A=4a^{q^2+q}b^q-a^{q^2}=a^{q^2}(4(ab)^q-1)\]
and
\[B=4ab^{q^2+1}-b^{q^2}=b^{q^2}(4ab-1).\]
If $4ab=1$, then
\[4a(x^{q+1}+ax^{2q}+bx^2)=4ax^{q+1}+4a^2x^{2q}+x^2=(2ax^q+x)^2.\]
This is planar on $\mathbb F_{q^3}$ if and only if $2ax^q+x$ is a permutation polynomial of $\mathbb F_{q^3}$, i.e., $\N(2a)\ne-1$ (this is equivalent to $\N(a)+\N(b)+ab\ne0$).

Suppose now $4ab\ne1$, which means $AB\ne0$. Then by Proposition \ref{AB}, the function $x^{q+1}+ax^{2q}+bx^2$ is planar if and only if
\[\frac{\N(A)+\N(B)}{AB}=4\N(a)+4\N(b)+1.\]
If $ab\in\mathbb F_q$, then
\[\N(A)+\N(B)=\N\big(a^{q^2}(4(ab)^q-1)\big)+\N\big(b^{q^2}(4ab-1)\big)=(\N(a)+\N(b))(4ab-1)^3,\]
and
\[AB=(ab)^{q^2}(4(ab)^q-1)(4ab-1)=ab(4ab-1)^2.\]
Consequently,
\[\begin{split}&\mathrel{\phantom{=}}\N(A)+\N(B)-AB(4\N(a)+4\N(b)+1)\\&=(\N(a)+\N(b))(4ab-1)^3-ab(4ab-1)^2(4\N(a)+4\N(b)+1)\\&=-(4ab-1)^2(\N(a)+\N(b)+ab).\end{split}\]
If $ab\notin\mathbb F_q$, then
\[AB=(ab)^{q^2}\frac{\N(4ab-1)}{4(ab)^{q^2}-1}=\frac{\N(4ab-1)}{4-(ab)^{-q^2}}\notin\mathbb F_q.\]
The results follow immediately.
\end{proof}

\section{Planar Trinomials on Quartic Extensions}\label{quartic}

In this section, let $\Tr$ and $\N$ denote the trace and the norm function from $\mathbb F_{q^4}$ to $\mathbb F_{q^2}$ respectively. A class of planar trinomials on $\mathbb F_{q^4}$ can be constructed in a similar way.

\begin{theorem}
For $a,b\in\mathbb F_{q^4}^*$, the function $x^{q^3+q}+ax^{q^2+1}+bx^2$ is planar on $\mathbb F_{q^4}$ if
\begin{enumerate}
\item[\upshape(1)] $\Tr(a)\ne0$, $\N(b)=\N(a)-\Tr(a)^{1-q}$, $1-4\Tr(a)^{-1-q}$ is a nonzero square in $\mathbb F_q$, and $\N(b)\Tr(a)$ is a nonzero square in $\mathbb F_{q^2}$, or
\item[\upshape(2)] $\Tr(a)=0$, $\N(b)=\N(a)-\theta^{1-q}$ for some $\theta\in\mathbb F_{q^2}^*$ such that $\theta^{\frac{q^2-1}2}=(-1)^{\frac{q-1}2}$, and $\N(b)\theta$ is a nonzero square in $\mathbb F_{q^2}$.
\end{enumerate} 
\end{theorem}
\begin{proof}
For the canonical additive character $\psi$ of $\mathbb F_{q^4}$, we have
\[\psi\big(cx^{q^3+q}+acx^{q^2+1}+bcx^2\big)=\psi\big((c^q+ac)x^{q^2+1}+bcx^2\big).\]
By Lemma \ref{nondegenerate}, the function $x^{q^3+q}+ax^{q^2+1}+bx^2$ is planar on $\mathbb F_{q^4}$ if and only if
\[\Tr(c^q+ac)^2-4\N(bc)\ne0\]
for every $c\in\mathbb F_{q^4}^*$. Suppose first that $a\in\mathbb F_{q^2}^*$, and let $d=\Tr(c)^2-4\N(c)$ and $t=\Tr(c)$. Then
\[\begin{split}\Tr(c^q+ac)^2-4\N(bc)&=\Tr(c)^{2q}+2a\Tr(c)^{q+1}+a^2\Tr(c)^2-4\N(bc)\\&=t^{2q}+2at^{q+1}+a^2t^2+\N(b)(d-t^2),\end{split}\]
and $\Tr(c^q+ac)^2-4\N(bc)=0$ if and only if
\[d=-\N(b)^{-1}(t^{2q}+2at^{q+1}+(a^2-\N(b))t^2).\]
Note that $d\ne0$ when $t=0$, $d=0$ when $c\in\mathbb F_{q^2}$ and $d$ is a non-square in $\mathbb F_{q^2}$ when $c\in\mathbb F_{q^4}\setminus\mathbb F_{q^2}$, so it suffices to show that
\[\N(b)(t^{2q}+2at^{q+1}+(a^2-\N(b))t^2)\]
is a nonzero square in $\mathbb F_{q^2}$ for every $t\in\mathbb F_{q^2}^*$.

Assume the condition (1); that is, $\N(b)=a^2-a^{1-q}$, $1-a^{-1-q}=\delta^2$ for some $\delta\in\mathbb F_q^*$, and $\N(b)a$ is a nonzero square in $\mathbb F_{q^2}$. Let $\alpha=\beta^{-q}a^{-1}$ for some $\beta\in\mathbb F_{q^2}$ such that $\beta^{q+1}=\delta+1$, where $\beta\ne0$, for otherwise $a^{-1-q}=1-\delta^2=0$. It follows that
\[\begin{split}\alpha^{q+1}+\beta^{q+1}&=\frac1{(\delta+1)a^{q+1}}+\delta+1\\&=\frac{1+(\delta+1)^2a^{q+1}}{(\delta+1)a^{q+1}}\\&=\frac{1+(1-a^{-1-q}+2\delta+1)a^{q+1}}{(\delta+1)a^{q+1}}\\&=2,\end{split}\]
and then
\[\begin{split}&\mathrel{\phantom{=}}\N(b)(t^{2q}+2at^{q+1}+(a^2-\N(b))t^2)\\&=\N(b)a(a^{-1}t^{2q}+2t^{q+1}+a^{-q}t^2)\\&=\N(b)a(\alpha\beta^qt^{2q}+(\alpha^{q+1}+\beta^{q+1})t^{q+1}+\alpha^q\beta t^2)\\&=\N(b)a(\alpha t^q+\beta t)^{q+1}.\end{split}\]
This is indeed a nonzero square in $\mathbb F_{q^2}$ for all $t\in\mathbb F_{q^2}^*$, since $\alpha^{q+1}\ne\beta^{q+1}$ (otherwise $\delta+1=\beta^{q+1}=1$), which means $\alpha t^q+\beta t\ne0$.

Suppose now that $a\notin\mathbb F_{q^2}$. Noting that
\[\Tr\left(\frac1{a-a^{q^2}}\right)=\Tr\left(a\frac{a^{q^2}}{a-a^{q^2}}\right)=0,\]
and
\[\Tr\left(a\frac1{a-a^{q^2}}\right)=\Tr\left(-\frac{a^{q^2}}{a-a^{q^2}}\right)=1,\]
we have
\[c=\frac1{a-a^{q^2}}\Tr(ac)-\frac{a^{q^2}}{a-a^{q^2}}\Tr(c)\]
for $c\in\mathbb F_{q^4}^*$. Thus,
\[\N(c)=\N\big(a-a^{q^2}\big)^{-1}(\Tr(ac)^2-\Tr(a)\Tr(ac)\Tr(c)+\N(a)\Tr(c)^2).\]
Let $t_0=\Tr(c)$, $t_1=\Tr(ac)$ and $\gamma=4\N(b)\N\big(a-a^{q^2}\big)^{-1}$, so that
\[\begin{split}&\mathrel{\phantom{=}}\Tr(c^q+ac)^2-4\N(bc)\\&=\Tr(c)^{2q}+2\Tr(c)^q\Tr(ac)+\Tr(ac)^2-4\N(bc)\\&=t_0^{2q}+2t_0^qt_1+t_1^2-\gamma(t_1^2-\Tr(a)t_0t_1+\N(a)t_0^2).\end{split}\]
Since $a\notin\mathbb F_{q^2}$, the map $c\mapsto(t_0,t_1)$ is bijective from $\mathbb F_{q^4}$ to $\mathbb F_{q^2}\times\mathbb F_{q^2}$, and thus $\Tr(c^q+ac)^2-4\N(bc)\ne0$ for all $c\in\mathbb F_{q^4}^*$ if and only if
\[(t_0^{2q}+2t_0^qt_1+t_1^2)-\gamma(t_1^2-\Tr(a)t_0t_1+\N(a)t_0^2)\ne0\]
for all $t_0,t_1\in\mathbb F_{q^2}$ with $t_0$ or $t_1$ nonzero. If $t_0=0$, then $\Tr(c^q+ac)^2-4\N(bc)=(1-\gamma)t_1^2$, and if $\gamma=1$, then $\Tr(c^q+ac)^2-4\N(bc)=0$ for all $t_1\in\mathbb F_{q^2}$. Therefore, assume $\gamma\ne1$ and then
\[\Tr(c^q+ac)^2-4\N(bc)=(1-\gamma)t_1^2+(2t_0^q+\gamma\Tr(a)t_0)t_1+t_0^{2q}-\gamma\N(a)t_0^2\]
can be viewed as a quadratic polynomial in $t_1$ over $\mathbb F_{q^2}$, whose discriminant is
\[\begin{split}D&=(2t_0^q+\gamma\Tr(a)t_0)^2-4(1-\gamma)(t_0^{2q}-\gamma\N(a)t_0^2)\\&=4\gamma t_0^{2q}+4\gamma\Tr(a)t_0^{q+1}+(\gamma^2\Tr(a)^2+4(1-\gamma)\gamma\N(a))t_0^2\\&=4\gamma t_0^{2q}+4\gamma\Tr(a)t_0^{q+1}+(\gamma^2(\Tr(a)^2-4\N(a))+4\gamma\N(a))t_0^2\\&=4\gamma t_0^{2q}+4\gamma\Tr(a)t_0^{q+1}+4\gamma(\N(a)-\N(b))t_0^2.\end{split}\]
Altogether, it suffices to show that $\gamma\ne1$ and $D$ is a non-square in $\mathbb F_{q^2}$ for all $t_0\in\mathbb F_{q^2}^*$.

Assuming the condition (1), we have $\Tr(a)\ne0$, $1-4\Tr(a)^{-1-q}=\delta^2$ for some $\delta\in\mathbb F_q^*$, and $\gamma\ne1$ since
\[4\N(b)-\N\big(a-a^{q^2}\big)=4\N(a)-4\Tr(a)^{1-q}+\big(a-a^{q^2}\big)^2=\Tr(a)^2(1-4\Tr(a)^{-1-q})\ne0.\]
Let $\alpha=\beta^{-q}\Tr(a)^{-1}$ for some $\beta\in\mathbb F_{q^2}$ such that $\beta^{q+1}=\frac{\delta+1}2$ ($\beta\ne0$ as easily seen). Then
\[\begin{split}\alpha^{q+1}+\beta^{q+1}&=\frac2{(\delta+1)\Tr(a)^{q+1}}+\frac{\delta+1}2\\&=\frac{4+(\delta+1)^2\Tr(a)^{q+1}}{2(\delta+1)\Tr(a)^{q+1}}\\&=\frac{4+(1-4\Tr(a)^{-1-q}+2\delta+1)\Tr(a)^{q+1}}{2(\delta+1)\Tr(a)^{q+1}}\\&=1.\end{split}\]
Hence,
\[\begin{split}&\mathrel{\phantom{=}}\Tr(a)^{-1}t_0^{2q}+t_0^{q+1}+\Tr(a)^{-1}(\N(a)-\N(b))t_0^2\\&=\Tr(a)^{-1}t_0^{2q}+t_0^{q+1}+\Tr(a)^{-q}t_0^2\\&=\alpha\beta^q t_0^{2q}+(\alpha^{q+1}+\beta^{q+1})t_0^{q+1}+\alpha^q\beta t_0^2\\&=(\alpha t_0^q+\beta t_0)^{q+1},\end{split}\]
and then
\[D=4\gamma\Tr(a)(\alpha t_0^q+\beta t_0)^{q+1}\]
is a non-square in $\mathbb F_{q^2}$ for all $t_0\in\mathbb F_{q^2}^*$, for
\[\gamma\Tr(a)=-4\big(a-a^{q^2}\big)^{-2}\N(b)\Tr(a)\]
is a non-square in $\mathbb F_{q^2}$ and $\alpha t_0^q+\beta t_0\ne0$ as $\alpha^{q+1}\ne\beta^{q+1}$.

Assume the condition (2). Then $\Tr(a)=0$ and $\gamma\ne1$ as
\[4\N(b)-\N\big(a-a^{q^2}\big)=4\N(b)-4\N(a)=-4\theta^{1-q}\ne0.\]
Moreover,
\[D=4\gamma t_0^{2q}+4\gamma(\N(a)-\N(b))t_0^2=4\gamma(t_0^{2q}+\theta^{1-q}t_0^2)=4\gamma\theta^{-q}(\theta^qt_0^{2q}+\theta t_0^2).\]
Here $\theta^qt_0^{2q}+\theta t_0^2\ne0$ for $t_0\in\mathbb F_{q^2}^*$, for otherwise $\theta t_0^2\notin\mathbb F_q$ and $-(\theta t_0^2)^{q+1}=(\theta t_0^2)^2$, which implies
\[(-1)^{\frac{q-1}2}\theta^{\frac{q^2-1}2}=(\theta t_0^2)^{q-1}\ne1,\]
contradicting the assumption. We conclude that $D$ is a non-square in $\mathbb F_{q^2}$ for all $t_0\in\mathbb F_{q^2}^*$, since
\[\gamma\theta^{-q}=-4\big(a-a^{q^2}\big)^{-2}\N(b)\theta^{-q}\]
is a non-square in $\mathbb F_{q^2}$ and $\theta^qt_0^{2q}+\theta t_0^2\in\mathbb F_q^*$.
\end{proof}

We show that this class of planar functions are not equivalent to monomials in general. For simplicity, let $q$ be a prime, so that we only need to consider the equivalence to $x^2$. Assume that
\[L_0\big(x^{q^3+q}+ax^{q^2+1}+bx^2\big)=L_1(x)^2\]
for some linearized permutation polynomials $L_0$ and $L_1$ over $\mathbb F_{q^4}$, and without loss of generality, let the coefficient of $x$ in $L_1(x)$ be nonzero. Comparing the degree, one easily finds that $L_1(x)=\alpha x^{q^2}+\beta x$ for some $\alpha,\beta\in\mathbb F_{q^4}$. Then
\[L_0\big(x^{q^3+q}+ax^{q^2+1}+bx^2\big)=\alpha^2x^{2q^2}+2\alpha\beta x^{q^2+1}+\beta^2x^2.\]
It follows that
\[L_0(bx^2)=\alpha^2x^{2q^2}+\beta^2x^2,\]
and thus,
\[L_0(x)=\alpha^2b^{-q^2}x^{q^2}+\beta^2b^{-1}x.\]
On the other hand,
\[L_0(1)x^{q^3+q}+L_0(a)x^{q^2+1}=L_0\big(x^{q^3+q}+ax^{q^2+1}\big)=2\alpha\beta x^{q^2+1},\]
which means $L_0(1)=0$, a contradiction.

\section*{Acknowledgement}

The work of the first author was supported by the China Scholarship Council. The funding corresponds to the scholarship for the PhD thesis of the first author in Paris, France. The French Agence Nationale de la Recherche partially supported the second author's work through ANR BARRACUDA (ANR-21-CE39-0009).


\end{document}